\documentclass[12pt, reqno]{amsart}
\usepackage{amsmath, amsthm, amscd, amsfonts, amssymb, graphicx, color}
\usepackage[bookmarksnumbered, colorlinks, plainpages]{hyperref}

\newtheorem{theorem}{Theorem}[section]
\newtheorem{lemma}[theorem]{Lemma}
\newtheorem{proposition}[theorem]{Proposition}
\newtheorem{corollary}[theorem]{Corollary}
\theoremstyle{definition}
\newtheorem{definition}[theorem]{Definition}

\theoremstyle{remark}
\newtheorem{remark}[theorem]{Remark}
\numberwithin{equation}{section}
\email{raissouli.mustapha@gmail.com}
\email{chergui\_m@yahoo.fr}
 
\begin{document}
\setcounter{page}{1}

\title[BiGamma Function and Inequalities]{The BiGamma Function and some of its Related Inequalities}

\author[M. Ra\"{\i}ssouli, M. Chergui]{Mustapha Ra\"{\i}ssouli$^1$ and Mohamed Chergui$^2$}
\address{$^1$ Department of Mathematics, Faculty of Science, Moulay Ismail University, Meknes, Morocco.}
\address{$^2$ Department of Mathematics, CRMEF, EREAM Team, LaREAMI-Lab, K\'enitra, Morocco.}



\date{Received: xxxxxx; Revised: yyyyyy; Accepted: zzzzzz.}

\begin{abstract}
In this article, we define a special function called the Bigamma function. It provides a generalization of Euler's gamma function. Several algebraic properties of this new function are studied. In particular, results linking this new function to the standard Beta function have been provided. We have also established inequalities, which allow to approximate this new function.
\end{abstract}

\maketitle

\section{\bf Introduction}

Special functions have particular properties that make them easier to manipulate and use in different contexts. For example, they may have special recurrence relations, integration, and derivation formulas that facilitate their use in complex calculations \cite{Ref2,Viola}. Examples of such functions include Bessel functions, Legendre functions, and hypergeometric functions. They are very useful for modeling physical phenomena, and, in particular, many special functions are solutions of certain differential equations \cite{BealsWong, Ref1}.

In this context, the gamma function and the beta function have been extensively developed for their interesting properties and diverse applications in many areas of mathematics and physics.

Gamma function is defined, for $x>0$, by
\begin{equation}\label{11}
\Gamma(x):=\int_0^\infty t^{x-1}e^{-t}\;dt. 
\end{equation}

This function has many applications in mathematical analysis, probability and statistics, as well as in physics, such as quantum mechanics. It appears in a wide variety of contexts and is a good tool for solving many scientific problems.

The gamma function satisfies the following functional relationship: $\Gamma(x+1)=x\Gamma(x)$ for all $x>0$. This last relation, which implies that $\Gamma(n+1)=n!$ for any integer $n\geq0$, tells us that the gamma function is an extension of the factorial function to non-integer numbers.

The beta function is defined, for $x>0, y>0$, by
\begin{equation}\label{12}
B(x,y):=\int_0^1 t^{x-1}(1-t)^{y-1}\;dt.
\end{equation}
There is an interesting relationship making in connection the beta function and the gamma function given by
\begin{equation}\label{14}
B(x,y)=\frac{\Gamma(x)\Gamma(y)}{\Gamma(x+y)}.
\end{equation}
For more details on the properties of these special functions, some of their extensions and their applications, we refer the interested reader to \cite{AAR,CHA,DAB,OOA,Raicher1,Raicher2,ZIC}.

By the change of variable $u=e^{-t}$, \eqref{11} becomes
\begin{equation}\label{15}
\Gamma(x)=\int_0^1\big(-\ln\;t\big)^{x-1}\;dt.
\end{equation}
By setting $t=1-u$, \eqref{15} can be also written as follows
\begin{equation}\label{17}
\Gamma(x)=\int_0^1\big(-\ln(1-t)\big)^{x-1}\;dt.
\end{equation}

It is well-known that the map $x\longmapsto\Gamma(x)$ is infinitely differentiable at any $x>0$ and its derivative at order $n\geq0$ is given by
$$\Gamma^{(n)}(x)=\int_0^\infty\big(\ln\;t\big)^nt^{x-1}e^{-t}\;dt,$$
or, equivalently, by the change of variable $u=e^{-t}$,
$$\Gamma^{(n)}(x)=\int_0^{1}\big(\ln(-\ln\;t)\big)^n\big(-\ln\;t\big)^{x-1}\;dt.$$
In particular, we have
\begin{equation}\label{13}
\Gamma^{(n)}(1)=\int_0^{1}\big(\ln(-\ln\;t)\big)^n\;dt=\int_0^{1}\big(\ln(-\ln(1-t))\big)^n\;dt.
\end{equation}

The fundamental aim of this article is to provide a generalization of the standard gamma function presented in \eqref{11} and to highlight its properties. Determining a link between this new generalization and certain classical functions is also targeted by the present work. These points will be developed in the following section, followed by some concluding remarks.  

\section{\bf Bigamma function}

Inspired by \eqref{15} and \eqref{17}, we consider the following improper integral
\begin{multline}\label{21}
\int_0^1\big(-\ln\;t\big)^{x-1}\big(-\ln(1-t)\big)^{y-1}\;dt=\int_0^{1/2}\big(-\ln\;t\big)^{x-1}\big(-\ln(1-t)\big)^{y-1}\;dt\\
+\int_{1/2}^1\big(-\ln\;t\big)^{x-1}\big(-\ln(1-t)\big)^{y-1}\;dt,
\end{multline}
whose convergence we will study. For this, we need the following lemma.

\begin{lemma}\label{lem1}
For any $\alpha>-1$ and $x>0$, we have
$$\int_0^1t^{\alpha}\big(-\ln\;t\big)^{x-1}\;dt=\frac{\Gamma(x)}{(\alpha+1)^{x}}.$$
\end{lemma}
\begin{proof}
Making the change of variable $t=e^{-u}$, we obtain
$$\int_0^1t^{\alpha}\big(-\ln\;t\big)^{x-1}\;dt=\int_0^1e^{-(\alpha+1)u}\;u^{x-1}\;du.$$
By applying the change of variable $(\alpha+1)u=s$ in this last integral, we deduce the desired result after a simple manipulation.
\end{proof}

Now, we will establish the convergence of the integral \eqref{21}. Indeed, we have $-\ln(1-t)\sim t$ and so $\big(-\ln(1-t)\big)^{y-1}\sim t^{y-1}$, when $t\longrightarrow0$. \\
By Lemma \ref{lem1}, we then deduce that the first integral in the second side of \eqref{21} is convergent for any $x>0,y>0$. The convergence of the last integral in \eqref{21} can be deduced when using the change of variable $t=1-u$.

After this, we may state the following main definition.

\begin{definition}
For $x>0,y>0$, the Bigamma function is defined by
\begin{equation}\label{22}
\Gamma(x,y):=\int_0^1\big(-\ln\;t\big)^{x-1}\big(-\ln(1-t)\big)^{y-1}\;dt.
\end{equation}
\end{definition}

It is clear that  $\Gamma(1,1)=1$, $\Gamma(x,y)=\Gamma(y,x)$ and $\Gamma(x,1)=\Gamma(1,x)=\Gamma(x)$. Furthermore, since $-\ln\;t>0$ and $-\ln(1-t)>0$ for any $t\in(0,1)$, then \eqref{22} immediately implies that $\Gamma(x,y)>0$ for any $x>0,y>0$. Other properties of $\Gamma(x,y)$ will be stated below.

\begin{proposition}For any $x>0, y>0$, there holds

$$\Gamma(x,y)=\int_0^\infty t^{x-1}e^{-t}\big(-\ln(1-e^{-t})\big)^{y-1}\;dt=\int_0^\infty t^{y-1}e^{-t}\big(-\ln(1-e^{-t})\big)^{x-1}\;dt.$$
\end{proposition}
\begin{proof}
The first relationship can be easily obtained by performing the change of variable $u=-\ln\,t$ in \eqref{22}, and the second one follows from the fact that $\Gamma(x,y)=\Gamma(y,x)$.
\end{proof}

\begin{proposition}
Let $a>0$ be a fixed real number. For any $x>0$ and $y>0$, we have
\begin{equation}\label{a}
\Gamma(x,y)=a^x\int_0^1t^{a-1}(-\ln\;t)^{x-1}\big(-\ln(1-t^a)\big)^{y-1}\;dt.
\end{equation}
\end{proposition}
\begin{proof}
Making the change of variable $t=u^a$, we deduce the desired result from \eqref{22} after a simple reduction.
\end{proof}

\begin{proposition}
Let $a>0$ with $a\neq1$, $x>0$ and $y>0$. We have the following assertions: 
\begin{itemize}
  \item[$i)$] If $x\leq1$ and  $a<1$, the following inequality holds
\begin{equation}\label{aa}
\Gamma(x,y)\geq \big(\frac{a}{e}\big)^{x-1}\Big(\frac{1-x}{1-a}\Big)^{x-1}\int_0^1t^{\frac{1-a}{a}}\Big(-\ln(1-t)\Big)^{y-1}\;dt.
\end{equation}
  \item[$ii)$] If $x\geq1$ and $a>1$, \eqref{aa} is reversed for $y>\frac{a-1}{a}$.
\end{itemize}

\end{proposition}
\begin{proof}
For fixed $x>0,a>0$, we set $\phi(t)=t^{a-1}(-\ln\;t)^{x-1}$. Simple computation leads to
$$\phi^{'}(t)=t^{a-2}(-\ln\;t)^{x-2}\Big((a-1)(-\ln\;t)-(x-1)\Big),$$ 
and if $a\neq1$ then, $\phi^{'}(t)=0$ if and only if $t=t_0:=e^{\frac{1-x}{a-1}}$. Under the assumption $x\leq1$ and $ a<1$, we get $t_0\in(0,1]$. With this, it is easy to see that the map $\phi$ realizes a minimum at $t_0$, with $\phi(t_0)=e^{1-x}\big(\frac{1-x}{1-a}\big)^{x-1}$.\\
By the use of \eqref{a}, we then deduce that
\begin{gather}\label{aa1}
  \Gamma(x,y)\geq a^xe^{1-x}\Big(\frac{1-x}{1-a}\Big)^{x-1}\int_0^1\big(-\ln(1-t^a)\big)^{y-1}\;dt.
\end{gather}

After this, we should justify that this latter improper integral is convergent. By similar arguments as in the previous results, we may check that such integral converges for any $y>0$ if $a<1$.\\ Now, if $x\geq1$ and $a>1$ then the map $\phi$ realizes a maximum at $t_0\in(0,1]$. In this case, the previous improper integral converges if $y>\frac{a-1}{a}$. Finally, making the change of variable $u=t^a$ in the integral in the right side of \eqref{aa1}, we get the desired result after simple operations.
\end{proof}

\begin{corollary}
Let $x\leq1$ and $y\geq1$. Then the following inequality
$$\Gamma(x,y)\geq\frac{a}{1+a(y-1)}\big(\frac{a}{e}\big)^{x-1}\Big(\frac{1-x}{1-a}\Big)^{x-1}$$
holds for any $0<a<1$. In particular, if $y=1$ then $\Gamma(x,1)=\Gamma(x)$ and we obtain
$$\forall\; x\leq1\;\;\; \forall\; 0<a<1\;\;\;\;\; \Gamma(x)\geq a^xe^{1-x}\Big(\frac{1-x}{1-a}\Big)^{x-1}.$$
\end{corollary}
\begin{proof}
For all $t\in(0,1)$ we have $-\ln(1-t)\geq t$ and so, $\big(-\ln(1-t)\big)^{y-1}\geq t^{y-1}$ if $y\geq1$. This, with a simple computation of integral yields
$$\int_0^1t^{\frac{1-a}{a}}\Big(-\ln(1-t)\Big)^{y-1}\;dt\geq\int_0^1t^{\frac{1-a}{a}}\;t^{y-1}\;dt=\frac{a}{1+a(y-1)}.$$
Substituting this in \eqref{aa}, we get the required inequality.
\end{proof}

\begin{remark}
By virtue of the relation $\Gamma(x,y)=\Gamma(y,x)$, analogous results to the previous ones can be stated when permuting $x$ and $y$. 
\end{remark}

The following result may be also stated.

\begin{proposition}\label{pr1}
The following assertions hold:
\begin{itemize}
  \item[$\bullet$] $\mbox{If}\;\; x\geq1, y\geq1\;\;\mbox{then}\;\; \Gamma(x,y)\geq\max\left(\frac{\Gamma(x)}{y^x},\frac{\Gamma(y)}{x^y}\right).$
  \item[$\bullet$] $\mbox{If}\;\; x\geq1, 0<y\leq1\;\;\mbox{then}\;\; \frac{\Gamma(y)}{x^y}\leq\Gamma(x,y)\leq\frac{\Gamma(x)}{y^x}.$
  \item[$\bullet$] $\mbox{If}\;\; 0<x\leq1, y\geq1\;\;\mbox{then}\;\; \frac{\Gamma(x)}{y^x}\leq\Gamma(x,y)\leq\frac{\Gamma(y)}{x^y}.$
  \item[$\bullet$] $\mbox{If}\;\; 0<x\leq1, 0<y\leq1\;\;\mbox{then}\;\; \Gamma(x,y)\leq\min\left(\frac{\Gamma(x)}{y^x},\frac{\Gamma(y)}{x^y}\right).$
\end{itemize}
\end{proposition}
\begin{proof}
For any $t\in(0,1)$ one has $-\ln(1-t)\geq t>0$. Then, $\big(-\ln(1-t)\big)^{y-1}\geq t^{y-1}$ if $y\geq1$, and $\big(-\ln(1-t)\big)^{y-1}\leq t^{y-1}$ if $y\leq1$.\\
To achieve the proof, we then use \eqref{22}, Lemma \ref{lem1} and the fact that $\Gamma(x,y)=\Gamma(y,x)$. The details are straightforward and therefore omitted here.
\end{proof}

The two following corollaries are immediate from the previous proposition and therefore we omit their proofs for the reader.

\begin{corollary}
If $x\geq1$ then we have $\Gamma(x,x)\geq\dfrac{\Gamma(x)}{x^x}$, with reversed inequality if\; $0<x\leq 1$.
\end{corollary}

\begin{corollary}
The following inequalities
$$x^{-1/x}\Gamma(\frac{1}{x})\leq\Gamma\big(x,\frac{1}{x}\big)\leq x^x\Gamma(x)$$
hold for any $x\geq1$, with reversed inequalities if\; $0<x\leq 1$.
\end{corollary}

We have also the following result.

\begin{proposition}
For any $x\geq1,y\geq1$ we have $\Gamma(x,y)\geq B(x,y)$, with reversed inequality if $x\leq1$ and $y\leq1$.
\end{proposition}
\begin{proof}
We use the same idea as in the proof of Proposition \ref{pr1} with the help of \eqref{12}.
\end{proof}

\begin{corollary}
For any $x\in(0,1)$, there holds
$$\Gamma(x,1-x)\leq\Gamma(x)\Gamma(1-x)=\frac{\pi}{\sin(\pi x)}.$$
\end{corollary}
\begin{proof}
The result follows from the previous proposition with the use of \eqref{14}.
\end{proof}

\begin{proposition}
For any $x>0,y>0$ there holds
$$\Gamma(x,y)\geq\exp\Big((x+y-2)\;\Gamma^{'}(1)\Big).$$
\end{proposition}
\begin{proof}
By Jensen's integral inequality and the concavity of the real function $x\longmapsto\ln\;x$, \eqref{22} yields
$$\ln \Gamma(x,y)\geq\int_0^1\ln\Big(\big(-\ln\;t\big)^{x-1}\big(-\ln(1-t)\big)^{y-1}\Big)\;dt.$$
This, with a simple algebraic manipulation and \eqref{13}, implies the desired result.
\end{proof}

From the previous proposition, we immediately deduce the following result.

\begin{corollary}
For any $x\in(0,1)$, the map $x\longmapsto\Gamma(x,1-x)$ is lower bounded with
$$\forall x\in(0,1)\;\;\;\;\; \Gamma(x,1-x)\geq e^{-\Gamma^{'}(1)}.$$
\end{corollary}

\begin{proposition}
The Bigamma function $(x,y)\longmapsto\Gamma(x,y)$ is logarithmically convex (so convex) on $(0,\infty)\times(0,\infty)$.
\end{proposition}
\begin{proof}
Let $p,q\geq0$ with $p+q=1$. For $x,y,a,b\geq0$ we have
\begin{multline*}
\Gamma\Big(p(x,y)+q(a,b)\Big):=\Gamma\Big(px+qa,py+qb\Big)\\
:=\int_0^1\big(-\ln\;t\big)^{px+qa-1}\big(-\ln(1-t)\big)^{py+qb-1}\;dt\\
=\int_0^1\big(-\ln\;t\big)^{p(x-1)+q(a-1)}\big(-\ln(1-t)\big)^{p(y-1)+q(b-1)}\;dt\\
=\int_0^1\Big(\big(-\ln\;t\big)^{x-1}\big(-\ln(1-t)\big)^{y-1}\Big)^p\Big(\big(-\ln\;t\big)^{a-1}\big(-\ln(1-t)\big)^{b-1}\Big)^q\;dt.
\end{multline*}
Applying the H\"{o}lder inequality to this latter expression we get
$$\Gamma\Big(p(x,y)+q(a,b)\Big)\leq\big(\Gamma(x,y)\big)^p\big(\Gamma(a,b)\big)^q,$$
so completing the proof.
\end{proof}

To provide more results we need the following lemma.

\begin{lemma}
Let $m,n\geq1$ be integers. Then the following integral converges
\begin{equation}\label{ln}
\Gamma_{ln}(m,n):=\int_0^1\Big(\ln\big(-\ln\;t\big)\Big)^{m-1}\;\Big(\ln\big(-\ln(1-t)\big)\Big)^{n-1}\;dt.
\end{equation}
\end{lemma}
\begin{proof}
We will prove that \eqref{ln} converges absolutely. Indeed, by the inequality $|a||b|\leq\frac{1}{2}a^2+\frac{1}{2}b^2$, valid for any real numbers $a$ and $b$, we can write
\begin{multline}\label{conv}
\Big|\Big(\ln\big(-\ln\;t\big)\Big)\Big|^{m-1}\;\Big|\Big(\ln\big(-\ln(1-t)\big)\Big)\Big|^{n-1}\\
\leq\frac{1}{2}\Big(\ln\big(-\ln\;t\big)\Big)^{2m-2}+\frac{1}{2}\Big(\ln\big(-\ln(1-t)\big)\Big)^{2n-2}.
\end{multline}
This, with the help of \eqref{13}, implies the required result.
\end{proof}

The following remark may be of interest and worth to be mentioned.

\begin{remark}
(i) From \eqref{ln}, it is obvious that $\Gamma_{ln}(1,1)=1$, $\Gamma_{ln}(m,n)=\Gamma_{ln}(n,m)$ and $\Gamma_{ln}(n,1)=\Gamma_{ln}(1,n)=\Gamma^{(n-1)}(1)$, for any integers $m,n\geq1$.\\
(ii) From \eqref{conv}, we deduce that
\begin{equation}\label{Y}
\Big|\Gamma_{ln}(m,n)\Big|\leq\frac{1}{2}\Gamma^{(2m-2)}(1)+\frac{1}{2}\Gamma^{(2n-2)}(1).
\end{equation}
(iii) The inequality \eqref{Y} can be improved when applying the H\"{o}lder inequality to \eqref{ln} and we get
$$\Big(\Gamma_{ln}(m,n)\Big)^2\leq\;\Gamma^{(2m-2)}(1)\;\Gamma^{(2n-2)}(1).$$
In particular, we have for any integer $n\geq1$
$$\Big|\Gamma_{ln}(n,n)\Big|\leq\Gamma^{(2n-2)}(1).$$
\end{remark}

Now, we are in the position to state the following result.

\begin{theorem}
For any $x>0,y>0$ there holds
\begin{equation}\label{sum}
\Gamma(x,y)=\sum_{m=n=0}^{\infty}\frac{(x-1)^m(y-1)^n}{m!\;n!}\;\Gamma_{ln}(m+1,n+1),
\end{equation}
or, equivalently,
$$\Gamma(x,y)= 1+ \big(x+y-2\big)\;\Gamma^{'}(1)+\sum_{m=n=1}^{\infty}\frac{(x-1)^m(y-1)^n}{m!\;n!}\;\Gamma_{ln}(m+1,n+1).$$
\end{theorem}
\begin{proof}
Substituting the following power series
\begin{multline*}
 (-\ln t)^{x-1}=\sum_{m=0}^\infty\frac{(x-1)^m}{m!}\Big(\ln\big(-\ln\;t\big)\Big)^{m} \text{ and } \\ \big(-\ln (1-t)\big)^{y-1}=\sum_{n=0}^\infty\frac{(y-1)^n}{n!}\Big(\ln\big(-\ln\,(1-t)\big)\Big)^{n},
\end{multline*}

in \eqref{22}, we get 
$$\Gamma(x,y)=\int_0^1 \sum_{m, n=0}^\infty\frac{(x-1)^m (y-1)^n}{n!\,m!}\Big(\ln\big(-\ln\;t\big)\Big)^{m}\,\Big(\ln\big(-\ln\,(1-t)\big)\Big)^{n}dt.$$
By noting that the power series involved are uniformly convergent, we can interchange the order of the integral with the infinite sum and taking in consideration the formula \eqref{ln}, we obtain \eqref{sum}.
\end{proof}

\begin{remark}
The formula \eqref{sum} appears as the power series of the Bigamma function $(x,y)\longmapsto\Gamma(x,y)$ at the point $(1,1)$. By an argument of uniqueness we can provide a relationship between $\Gamma_{ln}(m,n)$ and the partial derivatives of the Bigamma function $(x,y)\longmapsto\Gamma(x,y)$ at $(1,1)$. Precisely, we have
\begin{equation*}
\Gamma_{ln}(m,n)=\frac{\partial^{m+n-2}\Gamma}{\partial x^{m-1}\partial y^{n-1}}(1,1).
\end{equation*}
\end{remark}

\begin{corollary}
The following formula holds
$$\int_0^1\int_0^1\Gamma(x,y)\;dx\;dy=\sum_{m,n=1}^\infty\frac{(-1)^{m+n}}{m!\;n!}\;\Gamma_{ln}(m,n).$$
\end{corollary}
\begin{proof}
Follows by integrating side by side \eqref{sum} over $(0,1)\times(0,1)$. The details are simple and therefore omitted here.
\end{proof}

The following results provide some relationships between the beta function and the Bigamma function.

\begin{theorem}
Let $x>0,y>0$. Then the following expansion holds:
\begin{equation}\label{exp}
B(x,y)=\sum_{m,n=0}^\infty(-1)^{m+n}\frac{(x-1)^m(y-1)^n}{m!\;n!}\Gamma(m+1,n+1).
\end{equation}
\end{theorem}
\begin{proof}
 Substituting the following power series
\begin{equation*}
  t^{x-1}=\sum_{m=0}^\infty\frac{(x-1)^m}{m!}\big(\ln\,t\big)^{m} \text{ and } (1-t)^{y-1}=\sum_{n=0}^\infty\frac{(y-1)^n}{n!}\big(\ln\,(1-t)\big)^{n},
\end{equation*}

in \eqref{12}, we get
\begin{eqnarray*}
  B(x,y)&=& \int_0^1 \sum_{m, n=0}^\infty \frac{(x-1)^m (y-1)^n}{n!\,m!}\big(\ln\,t\big)^{m}\,\big(\ln\,(1-t)\big)^{n}dt \\
    &=& \int_0^1 \sum_{m, n=0}^\infty (-1)^{n+m}\,\frac{(x-1)^m (y-1)^n}{n!\,m!}\big(-\ln\,t\big)^{m}\,\big(-\ln\,(1-t)\big)^{n}dt. 
\end{eqnarray*}
  
Remarking that the power series involved are uniformly convergent, we can interchange the order of the integral with the infinite sum and taking into account the formula \eqref{22}, we obtain \eqref{exp}.\\
\end{proof}

\begin{corollary}
For any integers $m\geq1,n\geq1$ we have
$$\Gamma(m,n)=(-1)^{m+n}\frac{\partial^{m+n-2}}{\partial x^{m-1}\partial y^{n-1}}B(1,1)$$
\end{corollary}
\begin{proof}
As power series of $B(x,y)$ at $(1,1)$, \eqref{exp} implies that
$$\frac{\partial^{m+n}}{\partial x^{m}\partial y^n}B(1,1)=(-1)^{m+n}\Gamma(m+1,n+1).$$
The desired result follows after simple algebraic manipulations.
\end{proof}

\begin{corollary}
The following formula holds
$$\int_0^1\int_0^1B(x,y)\;dx\;dy=\sum_{m,n=0}^\infty\frac{\Gamma(m,n)}{m!\;n!}.$$
\end{corollary}
\begin{proof}
Integrating \eqref{exp} side by side over $(0,1)\times(0,1)$, with an argument of uniform convergence, we immediately obtain the required result after simple computations of integrals.
\end{proof}

We now introduce the following definition.

\begin{definition}
Let $x>0,y>0$. For $z\in[0,1]$, we define
\begin{equation}
\Gamma_z(x,y):=\int_0^z\big(-\ln\;t\big)^{x-1}\big(-\ln(1-t)\big)^{y-1}\;dt,
\end{equation}
which we call the incomplete Bigamma function.
\end{definition}

It is obvious that $\Gamma_0(x,y)=0$, $\Gamma_1(x,y)=\Gamma(x,y)$, $0\leq\Gamma_z(x,y)\leq\Gamma(x,y)$ and $\Gamma_z(1,1)=z$. We have also the following result.

\begin{proposition}
Let $x>0,y>0$. For any $z\in[0,1]$ the map $z\longmapsto\Gamma_z(x,y)$ is positive strictly increasing and satisfies
\begin{equation}\label{27}
\Gamma(x,y)=\Gamma_z(x,y)+\Gamma_{1-z}(y,x).
\end{equation}
\end{proposition}
\begin{proof}
The fact that $z\longmapsto\Gamma_z(x,y)$ is positive is obvious, and strictly increasing since $\frac{d}{dz}\Gamma_z(x,y)=\big(-\ln\;z\big)^{x-1}\big(-\ln(1-z)\big)^{y-1}>0$ for any $z\in(0,1)$. Now, putting $\Phi_{x,y}(t):=\big(-\ln\;t\big)^{x-1}\big(-\ln(1-t)\big)^{y-1}$ we have $\Phi_{x,y}(t)=\Phi_{y,x}(1-t)$. Writing
$$\Gamma(x,y)=\int_0^z\Phi_{x,y}(t)\;dt+\int_z^1\Phi_{y,x}(1-t)\;dt,$$
and making the change of variable $u=1-t$ in the second integral we immediately get the desired result.
\end{proof}

\begin{definition}
Let $x>0,y>0$. For $z\in[0,1]$, we set
\begin{equation}\label{29}
I_z(x,y):=\frac{\Gamma_z(x,y)}{\Gamma(x,y)},
\end{equation}
which will be called the incomplete Bigamma function ratio or the Bigamma distribution function.
\end{definition}

\begin{proposition}
Let $x>0,y>0$. The Bigamma distribution function $z\longmapsto I_z(x,y)$ satisfies the following relationship
\begin{equation}
I_z(x,y)=1-I_{1-z}(y,x).
\end{equation}
\end{proposition}
\begin{proof}
Follows from \eqref{27} and \eqref{29} with the fact that $\Gamma(x,y)=\Gamma(y,x)$.
\end{proof}
\section{\bf Conclusion}
Motivated by the importance of special functions in many fields, the present contribution focuses on the generalization of the standard gamma function. We have defined the Bigamma function $\Gamma(x,y)$ for two positive real numbers. To make this new function easier to use, we have formulated it in different formats. We were also able to establish some algebraic and geometric properties of this new special function. In particular, relations between the Bigamma function and certain classical functions have been established in this work. This leads us to think about investing this new function in particular areas in future research work.

\end{document}